\documentclass[12pt]{article}

\usepackage[latin1]{inputenc}
\usepackage[T1]{fontenc}

\usepackage{authblk} 

\usepackage{amsmath}
\usepackage{amsfonts}
\usepackage{amssymb}

\usepackage[normalem]{ulem} 

\usepackage{amsmath,xcolor,ulem}
\usepackage{amsfonts}
\usepackage{amssymb,todonotes}
\usepackage{amsthm}
\usepackage{graphicx}
\usepackage{marginnote}

\usepackage[colorlinks=true, allcolors=blue]{hyperref}

\usepackage[top=2.5cm,bottom=2.5cm,left=2.5cm,right=2.5cm]{geometry}
\usepackage{float}
\usepackage[algoruled]{algorithm2e}
\usepackage[font=footnotesize,width=\linewidth]{caption} 

\usepackage{graphicx}
\usepackage{amssymb}
\usepackage{amsthm}
\usepackage{amsmath}
\usepackage{lineno}
\usepackage{hyperref}
\usepackage{verbatim}

\newtheorem{theorem}{Theorem}[section]
\newtheorem{lemma}[theorem]{Lemma}
\theoremstyle{definition}

\newtheorem{example}[theorem]{Example}

\theoremstyle{remark}
\newtheorem{remark}[theorem]{Remark}

\numberwithin{equation}{section}

\usepackage{graphicx, pdflscape}
\usepackage{amsmath, hyperref}
 \usepackage{amssymb} 
\usepackage{relsize, float, lineno}

 \usepackage{xcolor} 
\usepackage{subcaption}
\allowdisplaybreaks

\title{A simple second-order nonstandard numerical method for a general class of dynamical systems and\\ its applications}

\author{Manh Tuan Hoang\footnote{Email(s): \href{mailto:tuanhm16@fe.edu.vn}{tuanhm16@fe.edu.vn}; \href{mailto:hmtuan01121990@gmail.com}{hmtuan01121990@gmail.com}}}

\affil{Department of Mathematics, FPT University, Hoa Lac Hi-Tech Park, \\ Km29 Thang Long Blvd, Hanoi, Viet Nam}

\begin{document}
\maketitle


\begin{abstract}
In this work,  we consider a class of continuous-time autonomous dynamical systems that model various important phenomena and processes encountered in real-world situations. We construct a second-order nonstandard finite difference (NSFD) method that simultaneously preserves two essential properties of the dynamical systems for all finite step sizes, namely the positivity of the solutions, the set of equilibrium points and their asymptotic stability. This NSFD method is constructed based on an appropriate choice of nonstandard denominator functions and a weighted discretization of the right-hand side functions. Under easily-verified conditions, the denominator functions guarantee second-order convergence, whereas the weights ensure the dynamic consistency. By taking advantage of the specific structure of the right-hand side functions, a simple discretization is utilized instead of the nonlocal discretization approaches commonly used in previous works. This simplifies the construction of the proposed NSFD method and, in particular, makes its asymptotic stability analysis easier.

As an illustration and an important application, we apply the constructed second-order NSFD method to a well-known two-stage structured species model with recruitment. Consequently, a simple second-order NSFD scheme for the considered two-stage structured species model is derived, improving upon a first-order NSFD scheme constructed in a previous work. Numerical experiments demonstrate the advantages of the second-order NSFD scheme over a standard second-order numerical method, namely, the explicit trapezoidal method.

The proposed NSFD method is simple and can be applied to a broad class of dynamical system models arising in both theory and applications.  Moreover, it can be readily combined with the Richardson extrapolation technique to improve its accuracy.
\end{abstract}

\begin{minipage}{0.9\linewidth}

\medskip


\end{minipage}

\section{Introduction}\label{intro}
We begin by considering a general dynamical system governed by ordinary differential equations (ODEs) of the form:
\begin{equation}\label{eq:DYN}
\dfrac{du(t)}{dt} = f(u(t)), \quad t \geq 0, \quad u(0) = u_0 \in \mathbb{R}^n,
\end{equation}
where $u$ is an $n$-component vector-valued function of $t$; $f = (f_1, f_2, \ldots,f_n)^T$ is a function of $u$ and it is assumed to satisfy suitable conditions ensuring the existence and uniqueness of solutions to the model \eqref{eq:DYN} (see, e.g., \cite{Allen, Khalil, Smith, Stuart}). In this work, we investigate \eqref{eq:DYN} under the following two assumptions, which commonly arise in a wide range of important mathematical models:
\begin{itemize}
\item[(\textbf{A1}):] There exists positive real number $c_i$ such that
\begin{equation}\label{eq:2new}
f_i(u) + c_iu_i \geq 0 \quad \mbox{for all} \quad u \geq 0.
\end{equation}
Here, the symbol '$\geq$' is understood in the componentwise (entry-wise) sense for vectors.
\item[(\textbf{A2}):] The equilibrium set is finite, and each equilibrium point is hyperbolic.
\end{itemize}
As a direct consequence of (\textbf{A1}), we conclude that \eqref{eq:DYN} admits the positive orthant $\mathbb{R}^n_+ = \big\{u \in \mathbb{R}^n| u \geq 0\big\}$ as a positively invariant set, that is, $u(t) \geq 0$ if $u_0 \geq 0$ (see \cite{Horvath, Smith}). Meanwhile,  (\textbf{A2}) implies that the (local) asymptotic stability of all the equilibria can be determined by the linearized method with the help of Routh-Hurwitz criteria \cite{Allen, Gantmacher}, that is, by examining the location of the eigenvalues of the Jacobian matrix evaluated at each equilibrium in relation to the left half of the complex plane \cite{Allen, Khalil, Stuart}. Specifically, we have
\begin{itemize}
\item[(i)] An equilibrium point $u^*$ is asymptotically stable if $Re(\lambda) < 0$ for all $\lambda \in \sigma(J(u^*))$, where $J(u^*) = \frac{\partial f}{\partial u}(u^*)$ and $\sigma(J)$ stands for the set of eigenvalues of $J$;
\item[(ii)] An equilibrium point $u^*$ is unstable if $Re(\lambda) > 0$ for some $\lambda \in \sigma(J(u^*))$.
\end{itemize}
It is easy to find several important mathematical models in biology, ecology and epidemiology, which satisfy (\textbf{A1}) and (\textbf{A2}), for instance:
\begin{itemize}
\item mathematical models in biology, epidemiology and chemostat \cite{Allen, Brauer, Martcheva, Smith};
\item a predator-prey model with linear prey growth and Beddington-DeAngelis functional response \cite{Dimitrov4};
\item vaccination models with multiple endemic states and  with non-linear incidence \cite{Gumel1, Kribs-Zaleta};
\item epidemic models with generalized non-linear incidence \cite{Capasso, Moghadas};
\item a mathematical model of Zika virus transmission \cite{Maamar};
\item an extended nonlinear three-compartmental model of ethanol metabolism in the human body \cite{Wacker}.
\end{itemize}
In \cite{Hoang2023}, a generalized NSFD method for dynamical system models satisfying (\textbf{A1}) and (\textbf{A2}) was proposed and analyzed. This method is based on a nonlocal approximation using weights for the right-hand side functions of the dynamical systems. It was rigorously established that the NSFD method is dynamically consistent with respect to the positivity, asymptotic stability, and three classes of conservation laws, namely, direct, generalized, and sub-conservation laws. However, the proposed NSFD method is only convergent of order $1$.

In recent years, the problem of constructing higher-order NSFD methods for differential equations has attracted considerable attention from researchers aiming to resolve the conflict between the dynamic consistency and the higher-order accuracy (see, for instance, see \cite{Alalhareth1, Alalhareth2, Hoang1, Hoang2, Hoang3, HoangMatthias1, HoangMatthias2, Kojouharov1} and references therein). These NSFD schemes are constructed based on extending Mickens' methodology \cite{Mickens1, Mickens2, Mickens3, Mickens4, Mickens5}, which employs nonlocal approximations of the right-hand side functions in combination with the renormalization of the denominator functions. Before, a class of second-order NSFD methods for ODEs with polynomial right-hand sides was introduced in \cite{Chen-Charpentier}; higher-order NSFD schemes using extrapolation techniques and variable step length algorithms for MSEIR and  SEIR epidemic models for malware propagation were formulated in \cite{Martin-Vaquero1, Martin-Vaquero2}. Another study combining NSFD schemes with Richardson extrapolation technique to improve the numerical solution of some population models can be found in an early work \cite{GParra}. In \cite{DangHoang}, explicit nonstandard Runge-Kutta methods, which have higher accuracy order and preserve the positivity and asymptotic stability of a class of autonomous dynamical systems, have been constructed based on the positivity of the Runge-Kutta methods. In recent work \cite{HoangMatthias2026},  a generalized, second-order, NSFD method for non-autonomous dynamical systems ODEs has been constructed. This method combines the NSFD framework with a new non-local approximation of the right-hand side function. It is worth nothing that the constructed NSFD methods avoids a restrictive and indispensable condition required by many existing positivity-preserving, second-order NSFD methods. Besides, an insight on some properties of high-order nonstandard linear multistep methods has been analyzed in \cite{Takacs}, whereas a general procedure to obtain unconditionally positive second-order NSFD methods has been provided in \cite{Conte}. It is safe to say that the construction of higher-order NSFD schemes is not a trivial problem, and the schemes developed mainly depend on the ODE models under consideration.

Our main objective in this work is to construct a simple second-order NSFD method for the dynamical system models of the form \eqref{eq:DYN}, which satisfy (\textbf{A1}) and (\textbf{A2}). Based on the approach proposed in \cite{HoangMatthias2}, we can drive a second-order NSFD scheme for \eqref{eq:DYN} in the form
\begin{equation}\label{eq:NSFDDYN}
\dfrac{u_i^{k + 1} - u_i^k}{\phi_i(h, u^k)} = P_i(u^k) - u_i^{k + 1}Q_i(u_k) + \tau_i u_i^k - \tau_i u_i^{k+1},
\end{equation}
where
\begin{itemize}
\item $u^k = \big(u_1^k,\, u_2^k,\,\ldots,\,u_n^k\big)^T$ is the intended approximation for $u(t_k) = \big(u_1(t_k),\, u_2(t_k),\,\ldots,\,u_n(t_k)\big)^T$ with $k = 1, 2, \ldots, N$;
\item $h = \dfrac{T}{N}$ ($N > 0$) is the step size;
\item $P_i$ and $Q_i$ ($i = 1,2,\,\ldots,\,n$) can considered as the positive and negative parts of $f_i$ and satisfy $P_i(u), Q_i(u) \geq 0$ and $P_i(u) - u_iQ_i(u) = f_i(u)$ for $u \geq 0$;
\item $\phi_i(h, u)$ are called a denominator function with the property that $0 < \phi_i(h, u) = h + \mathcal{O}(h^2)$ as $h \to 0$.
\item $\tau_i$ for $i = 1, 2, \ldots, n$ are positive real numbers, which play a role as weights.
\end{itemize}
The denominator functions $\phi_i(h, x, y)$ $(i = 1, 2)$ are chosen so that \eqref{eq:NSFDDYN} is convergent of order $2$, meanwhile, the weights $\tau_i$ ensure dynamic consistency. However, the stability analysis of \eqref{eq:NSFDDYN} becomes challenging because its complex structure. Therefore, we aim to construct a new NSFD method with a simpler structure. For this purpose, we adopt the approach in \cite{Hoang2023} to propose the following scheme:
\begin{equation}\label{eq:NSFDDYN1}
\dfrac{u_i^{k + 1} - u_i^k}{\phi_i(h, u^k)} = f_i(u^k) + \tau_i u_i^k - \tau_i u_i^{k+1},
\end{equation}
where $\tau_i$ for $i = 1, 2, \ldots, n$ are positive weights.

Note that if $\tau_i = 0$, we derive from \eqref{eq:NSFDDYN1} the nonstandard explicit Euler scheme, which was considered in \cite{Dimitrov, Kojouharov1}. Also, it should be emphasized that \eqref{eq:NSFDDYN1} has a simpler structure than \eqref{eq:NSFDDYN} because it does not rely on nonlocal discretizations of the right-hand side functions as was done in \cite{Alalhareth1, Alalhareth2, Conte, Hoang1, Hoang2, Hoang3, HoangMatthias1, HoangMatthias2, HoangMatthias2026}. This makes the mathematical analysis of \eqref{eq:NSFDDYN1} easier.  Also, \eqref{eq:NSFDDYN1} can be readily combined with the Richardson extrapolation technique \cite{Burden, Richardson, Roy} to improve its accuracy.

Through rigorous mathematical analysis, we establish suitable conditions imposed on $\tau_i$ and $\phi_i(h)$ ($i =1, 2, \ldots, n$), which ensure that \eqref{eq:NSFDDYN1} is second-order convergent as well as preserves the positivity and asymptotic stability of \eqref{eq:1} for all values of the step size.   


As an illustration and an important application, we consider a well-known two-stage structured species model with recruitment, which was first constructed in \cite{Ladino} and represented by
\begin{equation}\label{eq:1}
\begin{split}
\dfrac{dx(t)}{dt} &= \delta y(t)-\dfrac{\alpha x(t)}{\beta + x(t)}-\mu x(t) := f_1(x(t),\,y(t)),\\
\dfrac{dy(t)}{dt} &= \dfrac{\alpha x(t)}{\beta + x(t)}-(\mu+F) y(t) := f_2(x(t),\,y(t))
\end{split}
\end{equation}
subject to initial data $x(0), y(0) \geq 0$. In the model \eqref{eq:1}:
\begin{itemize}
\item the total fish population $N(t)$ is partitioned into two subpopulations: the pre-recruit group $x(t)$, consisting of eggs, larvae, and juveniles, and the exploitable group $y(t)$, consisting of adult fish;
\item all parameters are assumed to be positive in accordance with the biological characteristics of the system.
\end{itemize}
Further details of this model and its qualitative dynamics were fully characterized in the benchmark work of Ladino and Valverde \cite{Ladino}. In \cite{HoangValverde}, Hoang and Valverde extended Mickens' methodology \cite{Mickens1, Mickens2, Mickens3, Mickens4, Mickens5} to construct a first-order NSFD model for approximating the solutions of \eqref{eq:1}.

By applying the constructed second-order NSFD method  \eqref{eq:NSFDDYN1}, we drive a simple second-order NSFD scheme, which improves upon a first-order NSFD scheme constructed in \cite{HoangValverde}. It is worth noting that numerical experiments demonstrate advantages of the second-order NSFD scheme over a second-order standard numerical method, namely, the explicit trapezoidal method.

The rest of the paper is organized as follows:\\
Mathematical analysis of \eqref{eq:NSFDDYN1} is performed in Section \ref{Sec2}. The second-order NSFD scheme for the two-stage structured species model \eqref{eq:1} is constructed in Section \ref{Sec3}. Numerical simulations are conducted and reported in Section \ref{Sec4}. The final section contains some concluding remarks and discussions.
\section{Mathematical analysis of the proposed second-order NSFD method}\label{Sec2}
In this section, we determine conditions guaranteeing that \eqref{eq:NSFDDYN1} preserves the positivity and asymptotic stability of \eqref{eq:DYN} for all finite step sizes as well as is convergent of order $2$. For the sake of convenience, the denominator functions $\phi_i(h, u^k)$ will be written as $\phi_i$ in some places.

We first give a condition for the positivity of the solutions of \eqref{eq:NSFDDYN1}.
\begin{lemma}[The positivity of the NSFD method]\label{Lemma1P}
Let $\tau_i$ ($i = 1, 2, \ldots, n$) be real numbers with the property that
\begin{equation}\label{eq:GNSFD}
\tau_i \geq c_i.
\end{equation}
Then, \eqref{eq:NSFDDYN1} admits the positive orthant $\mathbb{R}_+^n$ as a positively invariant set, that is $u^k \geq 0$ whenever $u^0 \geq 0$.
\end{lemma}
\begin{proof}
This lemma is proved based on mathematical induction. Indeed, assume that $u^k \geq 0$ for $k \geq 0$. We need to show that $u^{k + 1} \geq 0$. It is easily verified that \eqref{eq:NSFDDYN1} can be transformed into the explicit form as
\begin{equation}\label{eq:GNSFD1}
u_i^{k + 1} = \dfrac{u^k_i + \phi_i\big(f_i(u^k) + \tau_iu_i^k\big)}{1 + \tau_i\phi_i}.
\end{equation}
It follows from \eqref{eq:2new} and \eqref{eq:GNSFD1} that
\begin{equation*}
f_i(u^k) + \tau_iu_i^k \geq -cu_i^k + \tau_iu^k = (\tau_i - c_i)u^k \geq 0.
\end{equation*}
Thus, under the condition \eqref{eq:GNSFD1}, $u^k \geq 0$ implies $u^{k + 1} \geq 0$. This is the desired conclusion. The proof is complete.
\end{proof}
Note that \eqref{eq:GNSFD1} can be rewritten in the form
\begin{equation}\label{eq:GNSFD2}
u_i^{k + 1} = u_i^k + \dfrac{\phi_i}{1 + \tau_i\phi_i}f_i(u^k).
\end{equation}
This implies that any equilibrium point $u^*$ of \eqref{eq:GNSFD1} satisfies the system
\begin{equation*}
f_i(u^*) = 0.
\end{equation*}
Hence, we obtain the following result.
\begin{lemma}[The set of equilibrium points]\label{Lemma2}
Let $\mathcal{F}_C$ and $\mathcal{F}_D^h$ be the sets of equilibrium points of \eqref{eq:DYN} and \eqref{eq:NSFDDYN1}, respectively. Then, $\mathcal{F}_C = \mathcal{F}_D^h$ for all $h > 0$.
\end{lemma}
We now analyze the asymptotic stability for \eqref{eq:NSFDDYN1}. It is easy to see that \eqref{eq:GNSFD2} can be represented in the vector form
\begin{equation}\label{eq:GNSFD3}
u^{k + 1} = u^k + D^k_hf^k,
\end{equation}
where
\begin{equation}\label{eq:GNSFD4}
u^k = 
\begin{bmatrix}
u_1^k\\
u_2^k\\
\vdots\\
u_n^k
\end{bmatrix},\quad
D_h^k = 
\begin{bmatrix}
\Psi_1&0&\ldots&0\\
0&\Psi_2&\ldots&0\\
\vdots&\vdots&&\vdots\\
0&0&\ldots&\Psi_n
\end{bmatrix},\quad
f^k =
\begin{bmatrix}
f_1(u^k)\\
f_2(u^k)\\
\vdots\\
f_n(u^k)
\end{bmatrix},
\end{equation}
where
\begin{equation}\label{eq:GNSFD5}
\Psi_i := \dfrac{\phi_i}{1 + \tau_i\phi_i}.
\end{equation}
It is easy to verify that $0 < \Psi_i = h + \mathcal{O}(h^2)$ as $h \to 0$. In other words, $\Psi_i$ defined in \ref{eq:GNSFD5} has the same property as $\phi_i$ but it contains $\tau_i$ as a control parameter.

With the support of Jury conditions or Schur-Cohn criteria \cite{Allen, Gantmacher}, we will determine appropriate conditions ensuring that \eqref{eq:NSFDDYN1} preserves the asymptotic stability of \eqref{eq:DYN}. Assume that $u^*$ is any equilibrium point of \eqref{eq:DYN} and its asymptotic stability has been determined via the linearized method with the help of Routh-Hurwitz criteria \cite{Allen}. We need to analyze its asymptotic stability with respect to \eqref{eq:NSFDDYN1}. The construction of the stability-preserving NSFD method now reduces to determining conditions such that
\begin{enumerate}
\item[(i)] if $u^*$ is asymptotically stable with respect to \eqref{eq:DYN}, then it is also asymptotically stable with respect to \eqref{eq:NSFDDYN1};
\item[(ii)] if $u^*$ is unstable with respect to \eqref{eq:DYN}, then it is also unstable with respect to \eqref{eq:NSFDDYN1}.
\end{enumerate}
By Jury conditions or Schur-Cohn criteria, the asymptotic stability of $u^*$ with respect to \eqref{eq:NSFDDYN1} is determined as:
\begin{itemize}
\item $u^*$ is asymptotically stable if $|\widehat{\lambda}| < 1$ for all $\widehat{\lambda} \in \sigma(J_D(u^*))$, where $J_D(u^*)$ is the Jacobian matrix of \eqref{eq:NSFDDYN1} evaluated at $u^*$;
\item $u^*$ is unstable if $|\widehat{\lambda}| > 1$ for some $\widehat{\lambda} \in \sigma(J_D(u^*))$.
\end{itemize}
Let us denote by $J_C(u^*)$ the Jacobian matrix of \eqref{eq:DYN} evaluated at $u^*$. Note that $u^*$ is also an equilibrium point of \eqref{eq:NSFDDYN1}. Then, it follows from \eqref{eq:GNSFD3} and \eqref{eq:GNSFD4} that the Jacobian of \eqref{eq:NSFDDYN1} evaluated at $u^*$ is given by
\begin{equation}\label{eq:GNSFD6}
J_D(u^*) = \mathbb{I} + D^k_hJ_C(u^*),
\end{equation}
where $\mathbb{I}$ is the identity matrix.

By using the methodology in \cite{Alalhareth1, Alalhareth2, Dimitrov5, Dimitrov6, Wood1}, we can determine stability thresholds $\Psi_i^S$ with the property that \eqref{eq:NSFDDYN1} preserves the asymptotic stability of the equilibrium points of \eqref{eq:DYN} whenever
\begin{equation*}
\Psi_i(h) < \Psi_i^S, \quad \forall 1 \leq i \leq n,
\end{equation*}
which is equivalent to
\begin{equation*}
\phi_i\big(1 - \tau_i\Psi_i^S\big) < \Psi_i^S.
\end{equation*}
This condition is satisfied if
\begin{equation}\label{eq:GNSFD7}
\tau_i \geq {\big(\Psi_i^S\big)}^{-1}.
\end{equation}
Thus, we determine a threshold of dynamic consistency $\tau_{DC}$ for \eqref{eq:NSFDDYN1}. More precisely, \eqref{eq:NSFDDYN1} preserves the posivitity of the solutions and the asymptotic stability of the equilibrium points of \eqref{eq:1} under the condition
\begin{equation}\label{eq:GNSFD8}
\tau \geq \tau_{DC}.
\end{equation}
We now construct a condition imposed on the denominator functions $\phi_i$ such that \eqref{eq:NSFDDYN1} is convergent of order $2$.
\begin{theorem}\label{Theorem2ndNSFD}
Let $\phi_i(h, u)$ ($i = 1, 2, \ldots, n$) be positive denominator functions that satisfy
\begin{equation}\label{eq:2ndGNSFD}
\dfrac{\partial^2 \phi_i(h, u)}{\partial h^2}\bigg|_{h = 0} = D_i(u) := 2\tau_i + \dfrac{1}{f_i(u)}\mathlarger{\mathlarger{\sum}}_{j = 1}^n\dfrac{\partial f_j}{\partial u_j}f_j(u),
\end{equation}
for all $u \geq \mathbb{R}_+^n$ such that $f_i(u) \ne 0$. Then, the truncation error of the NSFD scheme \eqref{eq:NSFDDYN1} is $\mathcal{O}(h^3)$, i.e., it is consistent of order $2$.
\end{theorem}
\begin{proof}
First, let us denote $R(h, u^k)$ the right-hand side function of \eqref{eq:GNSFD3}, that is $R(h, u^k) = u^k + D^k_hf^k$, where $D^k_h$ and $f^k$ are given in \eqref{eq:GNSFD4}. It is easy to verify that
\begin{equation}\label{eq:GNSFD9}
\begin{split}
R_i(0, u) &= u,\\
\dfrac{\partial R_i(h, u)}{\partial h}\bigg|_{h = 0} &= f_i(u),\\
\dfrac{\partial^2 R_i(h, u)}{\partial h^2}\bigg|_{h = 0} &= f_i(u)\bigg(\dfrac{\partial^2 \phi_i(h, u)}{\partial h^2}\bigg|_{h = 0} - 2\tau_i\bigg).
\end{split}
\end{equation}
Using Taylor's expansion in combination with \eqref{eq:GNSFD3} and \eqref{eq:GNSFD9} yields:
\begin{equation}\label{eq:GNSFD10}
\begin{split}
u_i^{k + 1} = R_i(h, u^k) &= R_i(0, u^k) + \dfrac{\partial R_i(h, u)}{\partial h}\bigg|_{h = 0}h + \dfrac{\partial^2 R_i(h, x, y)}{\partial h^2}\bigg|_{h = 0} \dfrac{h^2}{2} + \mathcal{O}(h^3),\\
&= u_i^k + f_i(u^k)h + f_i(u^k)\bigg(\dfrac{\partial^2 \phi_i(h, u)}{\partial h^2}\bigg|_{h = 0} - 2\tau_i\bigg)\dfrac{h^2}{2} + \mathcal{O}(h^3).
\end{split}
\end{equation}
On the other hand, using Taylor's expansion for the exact solution at $t_k$ gives
\begin{equation}\label{eq:GNSFD10a}
\begin{split}
u_i(t_{k + 1}) = u_i(t_k + h) &= u_i(t_k) + u_i'(t_k)h + u_i''(t_k)\dfrac{h^2}{2} + \mathcal{O}(h^3),\\
&= u_i(t_k) + f_i\big(u(t_k)\big)h + \mathlarger{\mathlarger{\sum}}_{j = 1}^n\dfrac{\partial f_j(u(t_k))}{\partial u_j}f_j(u(t_k))\dfrac{h^2}{2} + \mathcal{O}(h^3).
\end{split}
\end{equation}
Thus, it follows from \eqref{eq:GNSFD10} and \eqref{eq:GNSFD10a} that
\begin{equation*}
u_i(t_{k + 1}) - u_i^{k+1} = \mathcal{O}(h^3)
\end{equation*}
if \eqref{eq:2ndGNSFD} holds. This completes this proof.
\end{proof}
Following arguments analogous to those in the proof of \cite[Theorem 5.2]{Cresson} (see \cite[Appendix B]{Cresson}), it can be shown that the NSFD method \eqref{eq:NSFDDYN1} is second-order convergent under the assumptions of Theorem \ref{Theorem2ndNSFD}.
\begin{remark}
In a simple case, we can choose the denominator functions $\phi_i$ in the form
\begin{equation}\label{eq:18}
\phi_i(h, u)=\left\{\begin{array}{l}
\dfrac{e^{D_i(u)h} - 1}{D_i(u)} \quad \text { if } \quad D_i(u) \neq 0,\\
h \quad \text { if } \quad D_i(u)=0,
\end{array}\right.
\end{equation}
where $D_i(u)$ are defined in \eqref{eq:2ndGNSFD}. These denominator functions satisfy not only \eqref{eq:2ndGNSFD} but also $\phi_i(h, u) = h^2+ \mathcal{O}\left(h^2\right)$ as $h \rightarrow 0$ and $\phi_i(h, u) > 0$ for all $h >0, u \geq 0$.
\end{remark}
Before ending this section, we consider the combination of the constructed second-order NSFD method \eqref{eq:NSFDDYN1} with the Richardson extrapolation technique \cite{Burden, Richardson, Roy} to improve its accuracy. Let us denote by $\{u_{h}^k\}$ and $\{u_{h/2}^k\}$ the approximate solutions generated by the second-order NSFD method \eqref{eq:NSFDDYN1} by employing the step sizes $h$ and $h/2$, respectively. Then,
\begin{equation}\label{eq:3rdENSFD}
v_{h}^k := \dfrac{4u_{h/2}^k  - u_{h}^k}{3}
\end{equation}
generates an $\mathcal{O}(h^3)$ approximation formula for the solution of \eqref{eq:DYN} \cite{Burden}. Similarly, the fourth-order formula can be defined as  \cite{Burden}
\begin{equation}\label{eq:4thENSFD}
w^k_{h}:=\dfrac{8 v^k_{h/2}- v^k_{h}}{7}.
\end{equation}
Generally, higher-accuracy approximations can be generated from lower-accuracy ones. More clearly , the $\mathcal{O}(h^{p+1})$ approximate formula is given by  \cite{Burden}
\begin{equation*}
Y_{h}^k := \dfrac{2^p X^p_{h/2} - X^p_{h}}{2^p-1},
\end{equation*}
where $X^p_{h}$ is an $\mathcal{O}(h^p)$ approximate formula.
\begin{remark}\label{Remark1new}
In \cite{Maamar}, Maamar et al. constructed an NSFD scheme for a mathematical model of Zika virus transmission, which is represented by a relatively high-dimensional system of differential equations. However, it is easy to verify that this NSFD scheme is only convergent of order $1$. Similarly, Wacker in \cite{Wacker} proposed a first-order NSFD scheme for  an extended nonlinear three-compartmental model of ethanol metabolism in the human body. The NSFD method \eqref{eq:NSFDDYN1} can be applied to derive simple second-order and dynamically consistent NSFD schemes for these two ODE models.
\end{remark}
\section{Second-order NSFD scheme for the two-stage structured species model}\label{Sec3}
In this work, we apply the approach proposed in Section \ref{Sec2} to construct a second-order NSFD scheme for the two-stage (migratory) fish population model with recruitment \eqref{eq:1}.

In \cite{HoangValverde}, Hoang and Valverde extended Mickens' methodology \cite{Mickens1, Mickens2, Mickens3, Mickens4, Mickens5} to construct an NSFD model for approximating the solutions of \eqref{eq:1} on the time interval $[0, T]$. This NSFD scheme has the following form
\begin{equation}\label{eq:2}
\begin{split}
\dfrac{x_{k+1}- x_k}{\phi(h)} &= \delta y_k - \dfrac{\alpha x_{k+1}}{\beta + x_k}- \mu x_{k + 1}, \\
\dfrac{y_{k+1} - y_k}{\phi(h)} &= \dfrac{\alpha x_{k + 1}}{\beta + x_k} - (\mu+F) y_{k+1},
\end{split}
\end{equation}
where
\begin{itemize}
\item $(x_k,\,y_k)^T$ is the intended approximation for $(x(t_k),\,y(t_k))^T$ with $t_k = k\Delta t$ ($k = 1, 2, \ldots, N$) and $h= \dfrac{T}{N}$ being the step size;
\item $\phi(h)$ is a denominator function with the property that $0 < \phi(h) = h + \mathcal{O}(h^2)$.
\end{itemize}
Through rigorous mathematical analysis, simple conditions imposed on the denominator function were determined such that \eqref{eq:2} is dynamically consistent with \eqref{eq:1}. In particular, \eqref{eq:2} preserves the following properties of \eqref{eq:1} for any step sizes:
\begin{itemize}
\item[(\textbf{$P_1$})] \textbf{The positivity of the solutions:} The model \eqref{eq:1} admits the set
\begin{equation}\label{eq:3a}
\Omega = \{(x, y) \in \mathbb{R}^2|x, y \geq 0\}
\end{equation}
as a positively invariant set.
\item[(\textbf{$P_2$})]  \textbf{The set of equilibrium points:} A trivial equilibrium point $E_T^* = (0,\, 0)$ exists for all the values of the parameters, whereas  a non-trivial (positive) equilibrium point $E_P^* = (x^*,\,y^*)$ exists if and only if
\begin{equation*}
\mathcal{R}_0=\frac{\delta}{\mu+F}-\frac{\mu \beta}{\alpha} > 1.
\end{equation*}
Furthermore, when this is the case, $x^*$ and $y^*$ are computed as \cite{Ladino}
\begin{equation}\label{eq:3}
x^*=\dfrac{\alpha}{\mu}\left(\frac{\delta}{\mu+F}-\dfrac{\mu \beta}{\alpha}-1\right), \quad y^*=\dfrac{\alpha}{\delta-(\mu+F)}\left(\dfrac{\delta}{\mu+F}-\dfrac{\mu \beta}{\alpha}-1\right).
\end{equation} 
\item[(\textbf{$P_3$})]  \textbf{The asymptotic stability:} The trivial equilibrium point $E_T^*$ of \eqref{eq:1} is asymptotically stable of $\mathcal{R}_0 < 1$ and is unstable if $\mathcal{R}_0 > 1$. The positive equilibrium point $E_P^*$ is asymptotically stable if and only if it exists.
\end{itemize}
Although \eqref{eq:2} has the advantage of simulating the dynamics of \eqref{eq:1} over long time intervals due to its dynamic consistency and simplicity, it is only convergent of order $1$ (see \cite{HoangValverde}). Motivated by this, our main objective is to construct a second-order NSFD scheme that preserves the properties $(P_1)-(P_3)$ of \eqref{eq:1} for all values of the step size.

Based on the approach proposed in \cite{HoangMatthias2}, we can construct a second-order NSFD scheme for \eqref{eq:1} in the form
\begin{equation}\label{eq:4}
\begin{split}
\dfrac{x_{k+1}- x_k}{\phi_1(h, x_k, y_k)} &= \delta y_k - \dfrac{\alpha x_{k+1}}{\beta + x_k}- \mu x_{k + 1} + \tau_1x_k - \tau_1x_{k+1}, \\
\dfrac{y_{k+1} - y_k}{\phi_2(h, x_k, y_k)} &= \dfrac{\alpha x_{k + 1}}{\beta + x_k} - (\mu+F) y_{k+1} + \tau_2y_k - \tau_2y_{k+1},
\end{split}
\end{equation}
where $\tau_1$ and $\tau_2$ are positive real numbers, which play a role as weights and ensure the dynamic consistency of \eqref{eq:4}, whereas the denominator functions $\phi_i(h, x, y)$ $(i = 1, 2)$ are chosen so that \eqref{eq:4} is convergent of order $2$. However, the stability analysis of \eqref{eq:4} becomes challenging because it contains many parameters and the expression for the positive equilibrium is complicated. Therefore, we aim to construct a scheme with a simpler structure.

It is easy to verify that \eqref{eq:1} satisfies \eqref{eq:2new} with
\begin{equation*}
c_1 = \dfrac{\alpha}{\beta} + \mu, \quad c_2 = \mu + F.
\end{equation*}
Hence, it is reasonable to adopt the approach in Section \ref{Sec2} to derive the following scheme:
\begin{equation}\label{eq:NSFD}
\begin{split}
\dfrac{x_{k+1}- x_k}{\phi_1(h, x_k, y_k)} &= \delta y_k - \dfrac{\alpha x_{k}}{\beta + x_k}- \mu x_{k} + \tau_1x_k - \tau_1x_{k+1}, \\
\dfrac{y_{k+1} - y_k}{\phi_2(h, x_k, y_k)} &= \dfrac{\alpha x_{k}}{\beta + x_k} - (\mu+F) y_{k} + \tau_2y_k - \tau_2y_{k+1},
\end{split}
\end{equation}
$\tau_1$ and $\tau_2$ are positive real numbers.

%
%
%
In the next subsections, we will investigate dynamical properties and convergence analysis of the NSFD model \eqref{eq:NSFD}.  For this purpose, from now on we always assume that $\tau_1$ and $\tau_2$ satisfy:
\begin{equation}\label{eq:6}
\tau_1 \geq \dfrac{\alpha}{\beta} + \mu, \quad \tau_2 \geq \mu + F.
\end{equation}
\subsection{Basic properties}
We first examine the positively invariant set of \eqref{eq:NSFD}.
\begin{theorem}[Positively invariant set]\label{Theorem1}
Under the condition \eqref{eq:6}, the discrete-time model \eqref{eq:NSFD} admits the set $\Omega$ defined in \eqref{eq:3a} as a positively invariant set.
\end{theorem}
\begin{proof}
Assume that $x_0, y_0 \geq$. We need to show that $x_k, y_k \geq$ for all $k > 0$. Indeed, \eqref{eq:NSFD} can be transformed into the form
\begin{equation}\label{eq:7}
\begin{split}
x_{k + 1} &= \dfrac{x_k + \phi_1\delta y_k - \phi_1\dfrac{\alpha x_{k}}{\beta + x_k}- \phi_1\mu x_{k} + \phi_1\tau_1x_k}{1 + \tau_1\phi_1} = \dfrac{x_k + \phi_1\delta y_k - \phi_1\bigg(\dfrac{\alpha}{\beta + x_k} +  \mu - \tau_1\bigg)}{1 + \tau_1\phi_1}\\
y_{k + 1} &= \dfrac{y_k + \phi_2\dfrac{\alpha x_{k}}{\beta + x_k} - \phi_2(\mu+F) y_{k} + \phi_2\tau_2y_k}{1 + \tau_2\phi_2} = \dfrac{y_k + \phi_2\dfrac{\alpha x_{k}}{\beta + x_k} - \phi_2(\mu+F - \tau_2)y_k}{1 + \tau_2\phi_2}.
\end{split}
\end{equation}
It follows from \eqref{eq:6} that
\begin{equation*}
\begin{split}
&\dfrac{\alpha}{\beta + x_k} +  \mu - \tau_1 \leq \dfrac{\alpha}{\beta} + \mu - \tau_1 \leq 0,\\
&\mu+F - \tau_2 \leq 0.
\end{split}
\end{equation*}
Thus, we deduce from \eqref{eq:7} that $x_{k + 1}, y_{k + 1} \geq 0$ whenever $x_k, y_k \geq 0$. Hence, by mathematical induction, we conclude that $x_k, y_k \geq 0$ for $k > 0$ whenever $x_0, y_0 \geq 0$. This is desired conclusion. The proof is complete.
\end{proof}
It is easy to verify that \eqref{eq:7} can be rewritten in the form:
\begin{equation}\label{eq:NSFD1}
\begin{split}
x_{k+1} &= x_k + \dfrac{\phi_1}{1 + \tau_1\phi_1}f_1(x_k, y_k),\\
y_{k+1} &= y_k + \dfrac{\phi_1}{1 + \tau_2\phi_2}f_2(x_k, y_k),
\end{split}
\end{equation}

As a consequence of \eqref{eq:7}, we obtain that any equilibrium point of \eqref{eq:NSFD} is a solution of the system
\begin{equation*}
f_1(x, y) = f_2(x, y) = 0,
\end{equation*}
where $f_1$ and $f_2$ are defined in \eqref{eq:1}. This implies that the sets of the equilibrium points of \eqref{eq:NSFD} and \eqref{eq:1} are identical for any step sizes.

\subsection{Stability analysis of the trivial equilibrium point}\label{Subsec2.2}
This subsection investigates the asymptotic stability of the trivial equilibrium points with respect to the NSFD model \eqref{eq:NSFD}.

Let us denote by $J_C(E)$ the Jacobian matrix of \eqref{eq:1} evaluated at any equilibrium point $E$. It is easy to verify that
\begin{equation}\label{eq:9a}
J_C(E_T^*) =
\begin{pmatrix}
-\bigg(\dfrac{\alpha}{\beta} + \mu\bigg)&\delta\\
&\\
\dfrac{\alpha}{\beta}&-(\mu + F)
\end{pmatrix}
\end{equation}
and
\begin{equation}\label{eq:10a}
J_C(E_P^*) =
\begin{pmatrix}
-\dfrac{\alpha\beta}{(\beta + x_*)^2} - \mu&\delta\\
&\\
\dfrac{\alpha\beta}{(\beta + x_*)^2}& -(\mu + F)
\end{pmatrix}.
\end{equation}
From the mathematical analysis in \cite{Ladino}, we obtain
\begin{itemize}
\item If $\mathcal{R}_0 < 1$, then
\begin{equation}\label{eq:9}
\det(J^C(E_T^*)) > 0, \quad J_{11}^C(E_T^*) < 0, \quad J_{22}^C(E_T^*) < 0.
\end{equation}
and $\det(J^C(E_T^*)) < 0$ if $\mathcal{R}_0 > 1$.
\item If $\mathcal{R}_0 > 1$, then
\begin{equation}\label{eq:10}
\det(J^C(E_P^*)) > 0, \quad J_{11}^C(E_P^*) < 0, \quad J_{22}^C(E_P^*) < 0
\end{equation}
\end{itemize}
Note that \eqref{eq:9} and \eqref{eq:10} imply the asymptotic stability of $E_T^*$ and $E_P^*$ of the continuous-time model \eqref{eq:1}, respectively.
\begin{theorem}[Stability analysis of the trivial equilibrium point]\label{Theorem2}
(i) Assume that $\mathcal{R}_0 < 1$ and $\tau_1$ and $\tau_2$ are real numbers with the property that
\begin{equation}\label{eq:DC1}
\begin{split}
&\tau_2\big(-J_{11}^C(E_T^*)\big) + \tau_1\big(-J_{22}^C(E_T^*)\big) \geq \det(J^C(E_T^*)),\\
&\tau_1 \geq -\dfrac{J_{11}^C(E_T^*)}{2},\\
&\tau_2 \geq -\dfrac{J_{22}^C(E_T^*)}{2},\\
&C_0:= 4\tau_1\tau_2 + 2\tau_2J_{11}^C(E_T^*) + 2\tau_1J_{22}^C(E_T^*) \geq 0.
\end{split}
\end{equation}
Then, the trivial equilibrium point $E_T^*$ of \eqref{eq:NSFD} is asymptotically stable.\\
(ii) If $\mathcal{R}_0 > 1$, then the trivial equilibrium point $E_T^*$ of \eqref{eq:NSFD} is unstable.

\end{theorem}
\begin{proof}
\textbf{Proof of Part (i).} From \eqref{eq:7}, the Jacobian matrix of \eqref{eq:NSFD} evaluated at $E_T^0$ is given by
\begin{equation*}\label{eq:11}
J^D(E_T^*) = 
\begin{pmatrix}
1 + \dfrac{\phi_1}{1 + \tau_1\phi_1}J_{11}^C(E_T^*)&\dfrac{\phi_1}{1 + \tau_1\phi_1}J_{12}^C(E_T^*)\\
&\\
\dfrac{\phi_2}{1 + \tau_2\phi_2}J_{21}^C(E_T^*)&1 + \dfrac{\phi_2}{1 + \tau_2\phi_2}J_{22}^C(E_T^*)
\end{pmatrix}.
\end{equation*}
Consequently, the characteristic polynomial of $J^D(E_T^*)$ is
\begin{equation*}
P_{J^D(E_T^*)}(\lambda) = \lambda^2 - Tr(J^D(E_T^*))\lambda + \det(J^D(E_T^*)),
\end{equation*}
where
\begin{equation*}
\begin{split}
Tr(J^D(E_T^*)) &= 2  + \dfrac{\phi_1}{1 + \tau_1\phi_1}J_{11}^C(E_T^*) + \dfrac{\phi_2}{1 + \tau_2\phi_2}J_{22}^C(E_T^*),\\
\det(J^D(E_T^*)) &= \bigg(1 + \dfrac{\phi_1}{1 + \tau_1\phi_1}J_{11}^C(E_T^*)\bigg)\bigg(1 + \dfrac{\phi_2}{1 + \tau_2\phi_2}J_{22}^C(E_T^*)\bigg) - \dfrac{\phi_1}{1 + \tau_1\phi_1}J_{12}^C(E_T^*)\dfrac{\phi_2}{1 + \tau_2\phi_2}J_{21}^C(E_T^*).
\end{split}
\end{equation*}
By some algebraic manipulations, we obtain
\begin{equation}\label{eq:12}
\begin{split}
&\det(J^D(E_T^*)) = 1 + \dfrac{\phi_1}{1 + \tau_1\phi_1}J_{11}^C(E_T^*) + \dfrac{\phi_2}{1 + \tau_2\phi_2}J_{22}^C(E_T^*) + \dfrac{\phi_1\phi_2}{(1 + \tau_1\phi_1)(1 + \tau_1\phi_2)}\det(J^C(E_T^*)),\\
&1 - \det(J^D(E_T^*)) + Tr(J^D(E_T^*)) = \dfrac{\phi_1\phi_2}{(1 + \tau_1\phi_1)(1 + \tau_2\phi_2)}\det(J^C(E_T^*)),\\
&1 + \det(J^D(E_T^*)) + Tr(J^D(E_T^*)) = 4 + 2\dfrac{\phi_1}{1 + \tau_1\phi_1}J_{11}^C(E_T^*) + 2\dfrac{\phi_2}{1 + \tau_2\phi_2}J_{22}^C(E_T^*)\\
&+ \dfrac{\phi_1\phi_2}{(1 + \tau_1\phi_1)(1 + \tau_2\phi_2)}\det(J^C(E_T^*)).
\end{split}
\end{equation} 
We will show that
\begin{equation}\label{eq:13}
\begin{split}
&\det(J^D(E_T^*)) < 1,\\
&1 - \det(J^D(E_T^*)) + Tr(J^D(E_T^*))  > 0,\\
&1 + \det(J^D(E_T^*)) + Tr(J^D(E_T^*)) > 0.
\end{split}
\end{equation}

First, it is clear that $1 - \det(J^D(E_T^*)) + Tr(J^D(E_T^*))  > 0$ since $\det(J^D(E_T^*)) > 0$.

Second, using the first formula of \eqref{eq:12} gives
\begin{equation*}
\det(J^D(E_T^*)) - 1 = \dfrac{\phi_1J_{11}^C(E_T^*) + \phi_2J_{22}^C(E_T^*) + \phi_1\phi_2\big(\tau_2J_{11}^C(E_T^*) + \tau_1J_{22}^C(E_T^*) + \det(J^C(E_T^*))\big)}{(1 + \tau_1\phi_1)(1 + \tau_2\phi_2)},
\end{equation*}
which implies that $\det(J^D(E_T^*)) < 1$ if the first condition of \eqref{eq:DC1} holds.

Third, if follows from the third formula of \eqref{eq:12} that
\begin{equation*}
\begin{split}
1 + \det(J^D(E_T^*)) + Tr(J^D(E_T^*)) &= \dfrac{4 + \phi_1\phi_2\det(J^C(E_T^*))}{(1 + \tau_1\phi_1)(1 + \tau_2\phi_2)}\\
&+ \dfrac{2\phi_1\big(2\tau_1 + J_{11}^C(E_T^*)\big) + 2\phi_2\big(2\tau_2 + J_{22}^C(E_T^*)\big)}{(1 + \tau_1\phi_1)(1 + \tau_2\phi_2)}\\
&+ \dfrac{\phi_1\phi_2\big(4\tau_1\tau_2 + 2\tau_2J_{11}^C(E_T^*) + 2\tau_1J_{22}^C(E_T^*)\big)}{(1 + \tau_1\phi_1)(1 + \tau_2\phi_2)}.
\end{split}
\end{equation*}
From this, $1 + \det(J^D(E_T^*)) + Tr(J^D(E_T^*)) > 0$ if the last three conditions of \eqref{eq:DC1} are satisfied.

Thus, we have shown that \eqref{eq:13} occurs whenever \eqref{eq:DC1} holds. From a direct consequence of the Jury conditions or Schur-Cohn criteria \cite[Theorem 2.13]{Allen}, we conclude that the two eigenvalues $\widehat{\lambda}_1$ and $\widehat{\lambda}_2$ of $P_{J^D(E_T^*)}(\lambda)$ are strictly inside the unit circle, that is $|\widehat{\lambda}_i| < 1$. By the linearized method \cite[Theorem 1.3.7]{Stuart}, the asymptotic stability of $E_T^*$ is confirmed. The proof is complete.\\
\textbf{Proof of Part (ii).} Note that $\det(J^C(E_T^*)) < 0$ if $\mathcal{R}_0 > 1$. Hence, it follows from the second formula of \eqref{eq:12} that $1 - \det(J^D(E_T^*)) + Tr(J^D(E_T^*)) < 0$ if $\mathcal{R}_0 > 1$. Based on \cite[Theorem 2.13]{Allen} and \cite[Theorem 1. 3. 7]{Stuart}, we conclude that $E_T^*$ is unstable whenever $\mathcal{R}_0 > 1$. The proof is complete.
\end{proof}

In the following lemma, we simplify \eqref{eq:DC1} to a simpler system of conditions.
\begin{lemma}\label{Lemma1}
The system \eqref{eq:DC1} is satisfied whenever
\begin{equation}\label{eq:DC1new}
\begin{split}
&\tau_2\big(-J_{11}^C(E_T^*)\big) + \tau_1\big(-J_{22}^C(E_T^*)\big) \geq \det(J^C(E_T^*)),\\
&\tau_1\geq \tau_1^* := -J_{11}^C(E_T^*),\\
&\tau_2 \geq \tau_2^* := -J_{22}^C(E_T^*).
\end{split}
\end{equation}
\end{lemma}
\begin{proof}
First, we deduce from \eqref{eq:9} that \eqref{eq:DC1new} implies the first three conditions of \eqref{eq:DC1}. We only need to show that \eqref{eq:DC1new} implies the last condition of \eqref{eq:DC1}. Indeed, by setting $\tau_2 = q\tau_1$ with $q > 0$, the last condition of \eqref{eq:DC1} becomes
\begin{equation}
\begin{split}
C_0 &= 4\tau_1\tau_2 + 2\tau_2J_{11}^C(E_T^*) + 2\tau_1J_{22}^C(E_T^*)\\
&= 4q\tau_1^2 + 2q\tau_1J_{11}^C(E_T^*) + 2\tau_1J_{22}^C(E_T^*)\\
&= 2\tau_1\big(2q\tau_1 + q J_{11}^C(E_T^*) + J_{22}^C(E_T^*)\big),
\end{split}
\end{equation}
which implies that $C_0 \geq 0$ whenever $2q\tau_1 + q J_{11}^C(E_T^*) + J_{22}^C(E_T^*) \geq 0$. This is equivalent to
\begin{equation*}
\tau_1 \geq \tau_1^* := -\dfrac{qJ_{11}^C(E_T^*) + J_{22}^C(E_T^*)}{2q}.
\end{equation*}
This means that there always exist $\tau_1, \tau_2 > 0$ for which $C_0 \geq 0$. In particular, the last condition of \eqref{eq:DC1} can be rewritten in the form
\begin{equation*}
C_0 = 4\tau_1\tau_2 + 2\tau_2J_{11}^C(E_T^*) + 2\tau_1J_{22}^C(E_T^*) = 2\tau_1\big(\tau_2 + J_{22}^C(E_T^*)\big) + 2\tau_2\big(\tau_1 + J_{11}^C(E_T^*)\big),
\end{equation*}
which implies that $C_0 \geq 0$ if the last two conditions \eqref{eq:DC1new} hold. Consequently, the proof is complete.
\end{proof}

From Theorem \ref{Theorem2} and Lemma \ref{Lemma2}, we obtain a simple condition for \eqref{eq:NSFD} to be stability-preserving.
\begin{theorem}[Stability analysis of the trivial equilibrium point under a simplified condition]\label{Theorem3}
If $\mathcal{R}_0 < 1$ and $\tau_1$ and $\tau_2$ satisfies \eqref{eq:DC1new}, then the trivial equilibrium point $E_T^*$ of \eqref{eq:NSFD} is asymptotically stable. Furthermore, it is always unstable for any $\tau_1, \tau_2 > 0$ whenever $\mathcal{R}_0 > 1$.
\end{theorem}
\begin{remark}
From the proof of Lemma \ref{Lemma1}, we conclude that there exist positive real numbers $\tau_1^*$ and $\tau_2^*$ with the property that the system \eqref{eq:DC1} holds whenever $\tau_1 \geq \tau_1^*$ and $\tau_2 \geq \tau_2^*$.
\end{remark}
\subsection{Stability analysis of the positive equilibrium point}
This subsection investigates the asymptotic stability of the positive equilibrium point $E_P^*$ with respect to the NSFD model \eqref{eq:NSFD}. Assume that $E_P^*$ exists, that is $\mathcal{R}_0 > 1$. We observe from \eqref{eq:9} and \eqref{eq:10} that $J^C(E_T^*)$ in \eqref{eq:9a} and $J^C(E_P^*)$ in \eqref{eq:10a} share the same characteristic. Therefore, the stability analysis of the positive equilibrium point $E_P^*$ can be carried out similarly to that for the trivial equilibrium point $E_T^*$.

Based on  the mathematical analysis presented in Subsection \ref{Subsec2.2}, we establish the asymptotic stability of $E_P^*$ as follows.

\begin{theorem}[Stability analysis of the positive equilibrium point]\label{Theorem3}
Assume that $\mathcal{R}_0 > 1$. Let $\tau_1$ and $\tau_2$ be real numbers with the property that
\begin{equation}\label{eq:DC2}
\begin{split}
&\tau_2\big(-J_{11}^C(E_P^*)\big) + \tau_1\big(-J_{22}^C(E_P^*)\big) \geq \det(J^C(E_P^*)),\\
&\tau_1 \geq -\dfrac{J_{11}^C(E_P^*)}{2},\\
&\tau_2 \geq -\dfrac{J_{22}^C(E_P^*)}{2},\\
&C_*:= 4\tau_1\tau_2 + 2\tau_2J_{11}^C(E_P^*) + 2\tau_1J_{22}^C(E_P^*) \geq 0.
\end{split}
\end{equation}
\end{theorem}
%
\begin{theorem}[Stability analysis of the positive  equilibrium point under a simplified condition]\label{Theorem4}
Support that $\mathcal{R}_0 > 1$ and $\tau_1$ and $\tau_2$ are positive real numbers satisfying
\begin{equation}\label{eq:DC2new}
\begin{split}
&\tau_2\big(-J_{11}^C(E_P^*)\big) + \tau_1\big(-J_{22}^C(E_P^*)\big) \geq \det(J^C(E_P^*)),\\
&\tau_1\geq \widetilde{\tau}_1^* := -J_{11}^C(E_P^*),\\
&\tau_2 \geq \widetilde{\tau}_2^* := -J_{22}^C(E_P^*).
\end{split}
\end{equation}
Then the positive  equilibrium point $E_P^*$ of \eqref{eq:NSFD} is asymptotically stable.
\end{theorem}
\begin{remark}
By summarizing the results established in this section, we obtain thresholds of dynamic consistency for the proposed method \eqref{eq:NSFD}. More precisely, there exists positive real numbers $\tau_{DC1}$ and $\tau_{DC2}$ such that \eqref{eq:NSFD} is dynamically consistent with respect to the properties \textbf{$P_1$})-\textbf{$P_3$}) of \eqref{eq:1} whenever
\begin{equation*}
\tau_1 \geq \tau_{DC1}, \qquad \tau_2 \geq \tau_{DC2}.
\end{equation*}
Moreover, $\tau_{DC1}$ and $\tau_{DC2}$ can be computed easily.
\end{remark}
\subsection{Convergence analysis}
In this subsection, we determine conditions guaranteeing that the NSFD scheme \eqref{eq:NSFD} is convergent of order $2$.
\begin{theorem}\label{Theorem4}
Let $\phi_1(h, x, y)$ and $\phi_2(h, x, y)$ be positive denominator functions that satisfy
\begin{equation}\label{eq:2ndNSFD}
\begin{split}
\dfrac{\partial^2 \phi_1(h, x, y)}{\partial h^2}\bigg|_{h = 0} &= g_1(x, y)\\
&:= 2\tau_1 + \dfrac{-\bigg(\dfrac{\alpha}{(\beta + x)^2} + \mu\bigg)\bigg(\delta y- \dfrac{\alpha x}{\beta + x} - \mu x\bigg) + \delta\bigg(\dfrac{\alpha x}{\beta + x} - \mu y - Fy\bigg)}{\delta y - \dfrac{\alpha x}{\beta + x}-\mu x},\\
\dfrac{\partial^2 \phi_2}{\partial h^2}(0, x, y)\bigg|_{h = 0}  &= g_2(x, y)\\
&:= 2\tau_2 + \dfrac{\dfrac{\alpha}{(\beta + x)^2}\bigg(\delta y  - \dfrac{\alpha x}{\beta + x} - \mu x\bigg) - (\mu + F)\bigg(\dfrac{\alpha x}{\beta + x} - \mu y - Fy\bigg)}{\dfrac{\alpha x}{\beta + x}-(\mu+F) y}
\end{split}
\end{equation}
for all $(x, y) \geq \mathbb{R}_+^2$ such that $f_i(x, y) \ne 0$ for $i = 1, 2$. Then, the truncation error of the NSFD scheme \eqref{eq:NSFD} is $\mathcal{O}(h^3)$, i.e., it is consistent of order $2$.
\end{theorem}
\begin{proof}
First, let us denote
\begin{equation*}
\Psi_i(h, x, y) = \dfrac{\phi_i(h, x, y)}{1 + \tau_i\phi_i(x, y, h)}, \quad x, y \geq 0
\end{equation*}
for $i = 1, 2$. Then, \eqref{eq:NSFD} can be represented in the form
\begin{equation}\label{eq:14}
\begin{split}
&x_{k + 1} = x_k + \Psi_1(h, x_k, y_k)f_1(x_k, y_k) := R_1(h, x_k, y_k),\\
&y_{k + 1} = y_k + \Psi_2(h, x_k, y_k)f_2(x_k, y_k) := R_2(h, x_k, y_k),
\end{split}
\end{equation} 
where $f(x, y)$ and $g(x, y)$ are defined in \eqref{eq:1}. It is easy to verify that
\begin{equation}\label{eq:15}
\begin{split}
&R_1(0, x, y) = x,\\
&R_2(0, x, y) = y),\\
&\dfrac{\partial R_1(h, x, y)}{\partial h}\bigg|_{h = 0} = f_1(x, y),\\
&\dfrac{\partial R_2(h, x, y)}{\partial h}\bigg|_{h = 0} = f_2(x, y),\\
&\dfrac{\partial^2 R_1(h, x, y)}{\partial h^2}\bigg|_{h = 0} = f_1(x, y)\bigg(\dfrac{\partial^2 \phi_i(h, x, y)}{\partial h^2}\Big|_{h = 0} - 2\tau_1\bigg),\\
&\dfrac{\partial^2 R_2(h, x, y)}{\partial h^2}\bigg|_{h = 0} = f_2(x, y)\bigg(\dfrac{\partial^2 \phi_2(h, x, y)}{\partial h^2}\Big|_{h = 0} - 2\tau_2\bigg).
\end{split}
\end{equation}
Using Taylor's expansion in combination with \eqref{eq:14} and \eqref{eq:15} gives:\\
\begin{equation}\label{eq:16}
\begin{split}
x_{k + 1} &= R_1(h, x_k, y_k) = R_1(0, x_k, y_k) + \dfrac{\partial R_1(h, x, y)}{\partial h}\bigg|_{h = 0}h + \dfrac{\partial^2 R_1(h, x, y)}{\partial h^2}\bigg|_{h = 0} \dfrac{h^2}{2} + \mathcal{O}(h^3)\\
&= x_k + f_1(x_k, y_k)h + f_1(x, y)\bigg(\dfrac{\partial^2 \phi_1(h, x, y)}{\partial h^2}\bigg|_{h = 0} - 2\tau_1\bigg)\dfrac{h^2}{2} + \mathcal{O}(h^3),\\
y_{k + 1} &= R_2(h, x_k, y_k) = R_2(0, x_k, y_k) + \dfrac{\partial R_2(h, x, y)}{\partial h}\bigg|_{h = 0}h + \dfrac{\partial^2 R_2(h, x, y)}{\partial h^2}\bigg|_{h = 0} \dfrac{h^2}{2} + \mathcal{O}(h^3)\\
&= y_k + f_2(x_k, y_k)h + f_2(x, y)\bigg(\dfrac{\partial^2 \phi_2(h, x, y)}{\partial h^2}\bigg|_{h = 0} - 2\tau_2\bigg)\dfrac{h^2}{2} + \mathcal{O}(h^3).
\end{split}
\end{equation}
On the other hand, using Taylor's expansion for the exact solution at $t_k$ we obtain
\begin{equation}\label{eq:17}
\begin{split}
&x(t_{k + 1}) = x(t_k + h) = x(t_k) + x'(t_k)h + x''(t_k)\dfrac{h^2}{2} + \mathcal{O}(h^3),\\
&= x(t_k) + f_1\big(x(t_k),\,y(t_k)\big)h\\
&+ \bigg[-\bigg(\dfrac{\alpha}{(\beta + x(t_k))^2} + \mu\bigg)\bigg(\delta y(t_k) - \dfrac{\alpha x(t_k)}{\beta + x(t_k)} - \mu x(t_k)\bigg) + \delta\bigg(\dfrac{\alpha x(t_k)}{\beta + x(t_k)} - \mu y(t_k) - Fy(t_k)\bigg)\bigg]\dfrac{h^2}{2} + \mathcal{O}(h^3),\\
&y(t_{k + 1}) = y(t_k + h) = y(t_k) + y'(t_k)h + y''(t_k)\dfrac{h^2}{2} + \mathcal{O}(h^3),\\
&= y(t_k) + f_2\big(x(t_k),\,y(t_k)\big)h\\
&+ \bigg[\dfrac{\alpha}{\beta + x(t_k)}\bigg(\delta  - \dfrac{\alpha x(t_k)}{\beta + x(t_k)} - \mu x(t_k)\bigg) - (\mu + F)\bigg(\dfrac{\alpha}{\beta + x(t_k)} - \mu y(tk) - Fy(t_k)\bigg)\bigg]\dfrac{h^2}{2} + \mathcal{O}(h^3).
\end{split}
\end{equation}
Thus, it follows from \eqref{eq:16} and \eqref{eq:17} that if \eqref{eq:2ndNSFD} holds, then
\begin{equation*}
x(t_{k + 1}) - x_{k+1} = \mathcal{O}(h^3), \quad y(t_{k + 1}) - y_{k+1} = \mathcal{O}(h^3).
\end{equation*}
This is the desired conclusion and completes the proof.
\end{proof}
\section{Numerical experiments}\label{Sec4}
In this section, we conduct numerical examples to support the theoretical results. In the numerical examples reported below, the second-order NSFD scheme \eqref{eq:NSFD} with fixed weights $\tau_1$ and $\tau_2$ will be denote by $(\tau_1,\,\tau_2)$-2ndNSFD for simplicity. Also, the denominator functions $\phi_i(h)$ defined in \eqref{eq:18} will be used. 
\begin{example}[An error analysis of the NSFD scheme]
In this example, we provide an error analysis for the constructed second-order NSFD scheme \eqref{eq:NSFD}. To end this, we consider \eqref{eq:1} with the following set of the parameters:
\begin{equation*}\label{eq:19}
\alpha = 20, \quad \beta = 10, \quad \mu = 0.897, \quad F = 4.55 - \mu, \quad \delta = 1.05,
\end{equation*}
and the initial data $x(0) = 100$ and $y(0) = 90$.

For this set, we obtain $\mathcal{R}_0 = -0.2177 < 1$. Hence, the trivial equilibrium point $E_T^* = (0,\,0)$ is asymptotically stable. By some algebraic manipulations, \eqref{eq:DC1new} is simplified to
\begin{equation*}
\begin{split}
&\tau_1 \geq 2.8970,\\
&\tau_2 \geq 4.5500,\\
&2.8970 \tau_1 + 4.5500 \tau_2 \geq 11.0814.
\end{split}
\end{equation*}
Consequently, we can choose $(\tau_1,\,\tau_2) = (2.8970,\,4.5500)$.

To estimate errors generated by the NSFD scheme \eqref{eq:NSFD} over the interval $[0,\,1]$, we admit the numerical approximation, which is obtained by applying a $11$-stage Runge-Kutta method of order $8$ (see \cite{Copper}) with a step size $h = 10^{-6}$ as a reference solution. Then, the errors are computer as

\begin{equation*}
\begin{split}
&err_M = \max_{k}e_k, \quad e_k := |x_k - x(t_k)| + |y_k - y(t_k)|, \,\,\, t_k = kh,\,\,\,h=\dfrac{1}{N},\,\,\,0 \leq k \leq N,\\
&err_F = e_N, \\
&err_A = \dfrac{\mathlarger{\sum}_{k = 1}^Ne_k}{N}.
\end{split}
\end{equation*}
Besides, the rate of convergence (ROC) is estimated  by (see \cite{Ascher}):
\begin{equation*}
ROC := \log_{\bigg(\dfrac{h_1}{h_2}\bigg)}\bigg(\dfrac{err_F(h_1)}{err_F(h_2)}\bigg).
\end{equation*}

The errors and ROC corresponding to the NSFD scheme \eqref{eq:NSFD} with some different values of $\tau_1$ and $\tau_2$ are reported in Table \ref{Table1}--\ref{Table3}, whereas the errors and ROC generated by the explicit trapezoidal method (see \cite{Ascher}) and the first-order NSFD scheme (1stNSFD) \eqref{eq:2} are presented in Tables \ref{Table5}, and \ref{Table4}, respectively. Besides, the errors generated by the second-order NSFD and trapezoidal schemes with $h = 0.01$ are depicted in Figure \ref{Fig:7}.

It is clear that the constructed NSFD method is convergent of order $2$ as the trapezoidal method and they improves the first-order NSFD scheme \eqref{eq:2}. However, the errors produced by these methods are different. Although the trapezoidal method yields the smallest final-time error, its maximum and average errors are larger than those of the second-order schemes $(2.8970,\,4.5500)$-2ndNSFD and $(3.0000,\,5.000)$-2ndNSFD. Furthermore, Tables \ref{Table1}--\ref{Table3} illustrate the influence of $(\tau_1,\,\tau_2)$ on the errors of the second-order NSFD schemes. It is worth noting that the values of $\tau_1$ and $\tau_2$ in Tables \ref{Table1} and \ref{Table2} satisfy the dynamic consistency condition \eqref{eq:DC1new}, whereas those in Table \ref{Table3} do not. This highlights the importance of determining optimal weights to achieve the best error performance.

Figure \ref{fig:1} and \ref{fig:2} represent approximate solutions generated by the trapezoidal method with $h = 0.4$ and by the $(2.8970,\,4.5500)$-2ndNSFD scheme with some different values of $h$. We observe from these figures that the trapezoidal method fails to preserve the positivity and asymptotic stability of continuous-time model with $h = 0.4$. However, the second-order NSFD scheme preserves these properties regardless of the step sizes used. This is an advantage of NSFD schemes over standard ones, which has been demonstrated in previous studies \cite{Mickens1, Mickens2, Mickens3, Mickens4, Mickens5}.

Tables \ref{Table6} and \ref{Table7} report the errors and ROC of the third-order and fourth-order extrapolated $(2.8970,\,4.5500)$-2ndNSFD schemes, which are derived from \eqref{eq:3rdENSFD} and \eqref{eq:4thENSFD}, respectively. Clearly, the accuracy of the underlying second-order NSFD scheme can be easily improved by combining them with Richardson extrapolation technique. The similar observations can be found in previous studies \cite{GParra, Martin-Vaquero1, Martin-Vaquero2}.
\begin{table}[H]
\begin{center}
\caption{The errors and ROC of the $(2.8970,\,4.5500)$-2ndNSFD scheme}\label{Table1}
\begin{tabular}{cccccccccc}
\hline
$h$&$err_M$&$err_F$&$err_A$&ROC\\
\hline
$10^{-1}$&0.6116&0.6116&0.3172&\\
$10^{-2}$&0.0061&0.0047& 0.0032&2.1187\\
$10^{-3}$&5.4086e-005&4.8969e-005&3.0555e-005&1.9779\\
$10^{-4}$&5.3488e-007& 4.9190e-007&3.0472e-007&1.9980\\
$10^{-5}$&5.3416e-009&4.9223e-009&3.0465e-009&1.9997\\
$10^{-6}$&5.4541e-011&4.8480e-011&3.0413e-011&2.0066\\
\hline
\end{tabular}
\end{center}
\end{table}
\begin{table}[H]
\begin{center}
\caption{The errors and ROC of the $(3.0000,\,5.0000)$-2ndNSFD scheme}\label{Table2}
\begin{tabular}{cccccccccc}
\hline
$h$&$err_M$&$err_F$&$err_A$&ROC\\
\hline
$10^{-1}$&0.8237&0.8237&0.5280&\\
$10^{-2}$&0.0077&0.0066&0.0050&2.0952\\
$10^{-3}$&7.1013e-005&6.8590e-005&4.9663e-005&1.9843\\
$10^{-4}$&7.0431e-007&6.8812e-007&4.9653e-007& 1.9986\\
$10^{-5}$&7.0355e-009&6.8843e-009&4.9652e-009&1.9998\\
$10^{-6}$& 7.2887e-011&7.0392e-011& 5.1043e-011&1.9903\\
\hline
\end{tabular}
\end{center}
\end{table}
\begin{table}[H]
\begin{center}
\caption{The errors and ROC of the $(1.6000,\,2.0000)$-2ndNSFD scheme}\label{Table3}
\begin{tabular}{cccccccccc}
\hline
$h$&$err_M$&$err_F$&$err_A$&ROC\\
\hline
$10^{-1}$&2.2826&0.9630&1.5098&\\
$10^{-2}$&0.0149&0.0060&0.0107&2.2053\\
$10^{-3}$&1.4110e-004&5.6175e-005&1.0102e-004&2.0288\\
$10^{-4}$&1.4037e-006&5.5829e-007&1.0053e-006&2.0027\\
$10^{-5}$&1.4021e-008&5.5481e-009&1.0042e-008& 2.0027\\
$10^{-6}$&2.4276e-010&1.0126e-010&9.8906e-011&1.7387\\
\hline
\end{tabular}
\end{center}
\end{table}
\begin{table}[H]
\begin{center}
\caption{The errors and ROC of the trapezoidal method}\label{Table5}
\begin{tabular}{cccccccccc}
\hline
$h$&$err_M$&$err_F$&$err_A$&ROC\\
\hline
$10^{-1}$& 1.9686&0.2343&0.9915&\\
$10^{-2}$&0.0142&0.0018&0.0076&2.1225\\
$10^{-3}$&1.3795e-004&1.7306e-005&7.4680e-005&2.0091\\
$10^{-4}$&1.3752e-006&1.7270e-007& 7.4503e-007&2.0009 \\
$10^{-5}$&1.3744e-008&1.7265e-009&7.4475e-009&2.0001\\
$10^{-6}$&1.3673e-010&1.6443e-011&7.4537e-011&2.0212\\
\hline
\end{tabular}
\end{center}
\end{table}
\begin{table}[H]
\begin{center}
\caption{The errors and ROC of the 1stNSFD scheme \eqref{eq:2} with $\phi(h) = 1 - e^{-h}$}\label{Table4}
\begin{tabular}{cccccccccc}
\hline
$h$&$err_M$&$err_F$&ROC\\
\hline
$10^{-1}$& 12.9500& 10.3584&\\
$10^{-2}$&1.4375& 1.0513& 1.1828&0.9936\\
$10^{-3}$&0.1455&0.1052&0.1200&0.9995\\
$10^{-4}$&0.0146& 0.0105&0.0120&1.0000\\
$10^{-5}$&0.0015&0.0011&0.0012&1.0000\\
$10^{-6}$&1.4566e-004& 1.0525e-004&1.2016e-004&1.0000\\
\hline
\end{tabular}
\end{center}
\end{table}
\begin{table}[H]
\begin{center}
\caption{The errors and ROC of the third-order extrapolated $(2.8970,\,4.5500)$-2ndNSFD scheme}\label{Table6}
\begin{tabular}{cccccccccc}
\hline
$h$&$err_M$&$err_F$&ROC\\
\hline
$0.1$&0.1609&0.0822&\\
$0.05$&0.0129&0.0045&4.1968\\
$0.025$&0.0022&8.1494e-004& 2.4589\\
$0.01$&1.2662e-004&4.6659e-005&3.1215 \\
$0.005$&1.4920e-005&5.4508e-006&3.0976  \\
$0.0025$&1.8052e-006&6.5613e-007&3.0544\\
$0.001$& 1.1321e-007&4.1012e-008&3.0258\\
$0.0001$&1.1403e-010&4.0616e-011&3.0042\\
\hline
\end{tabular}
\end{center}
\end{table}
\begin{table}[H]
\begin{center}
\caption{The errors and ROC of the fourth-order extrapolated $(2.8970,\,4.5500)$-2ndNSFD scheme}\label{Table7}
\begin{tabular}{cccccccccc}
\hline
$h$&$err_M$&$err_F$&ROC\\
\hline
$0.1$&0.0375&0.0169\\
$0.05$&6.5445e-004&2.9127e-004&5.8550 \\
$0.025$& 2.5011e-005&9.1939e-006&4.9855\\
$0.01$&1.0853e-006&4.3609e-007&3.3269\\
$0.005$&7.1040e-008&2.8827e-008& 3.9191 \\
$0.0025$&4.4774e-009&1.8238e-009& 3.9824\\
$0.001$&1.1607e-010&4.6278e-011&4.0097\\
$0.0001$&5.8265e-012&1.5898e-013&2.4640\\
\hline
\end{tabular}
\end{center}
\end{table}
\begin{figure}[H]
\centering
\includegraphics[height=10.0cm,width=16cm]{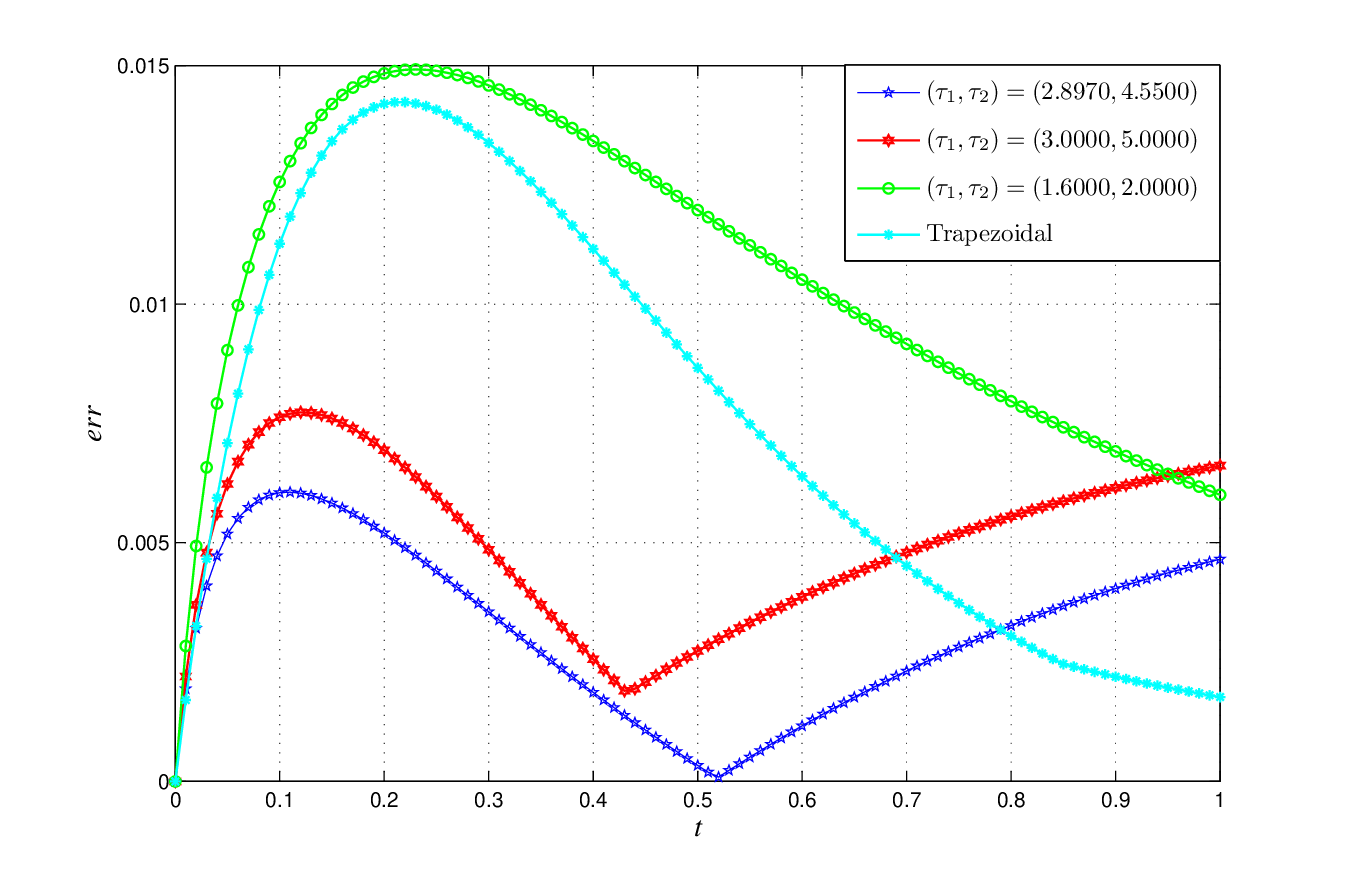}
\caption{The errors provided by the second-order NSFD and trapezoidal methods}\label{Fig:7}
\end{figure}
\begin{figure}[H]
\subfloat[$x$-component]{%
\includegraphics[height=10cm,width=15cm]{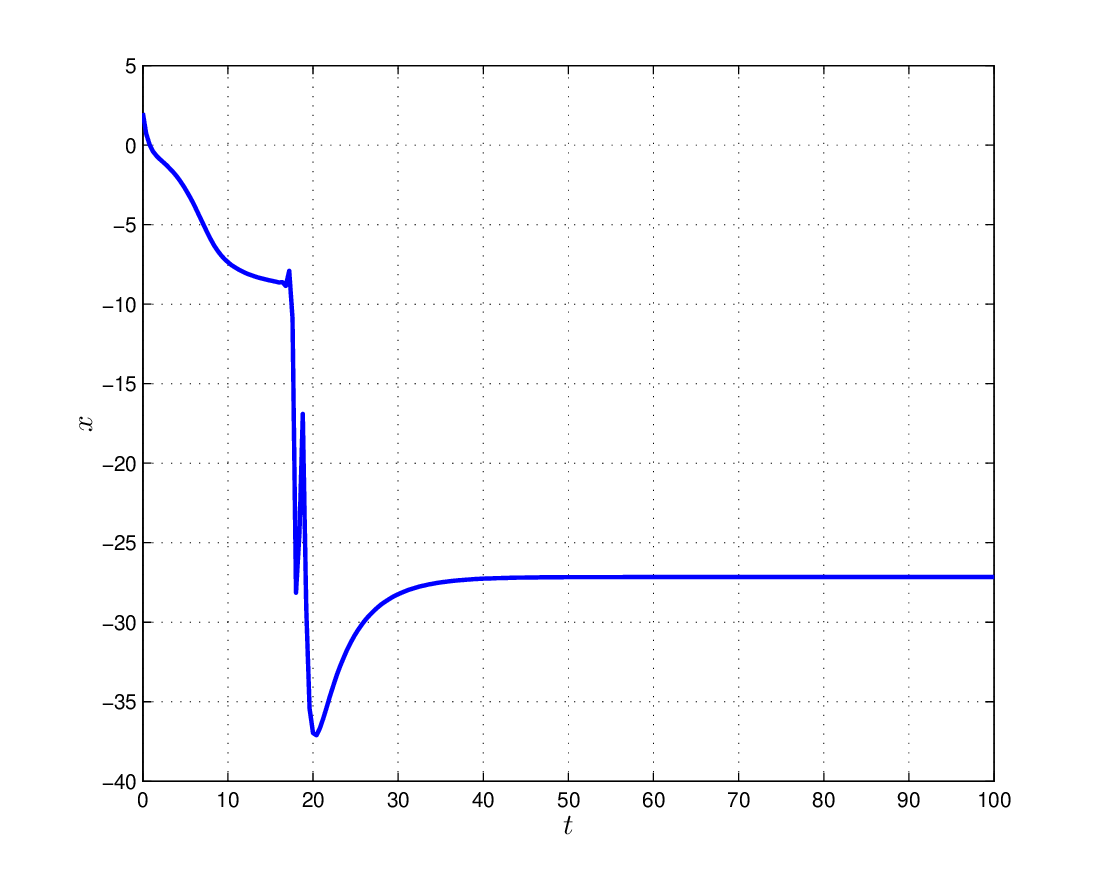}
\label{Figure:1a}
}\hfill
\subfloat[$y$-component]{%
\includegraphics[height=10cm,width=15cm]{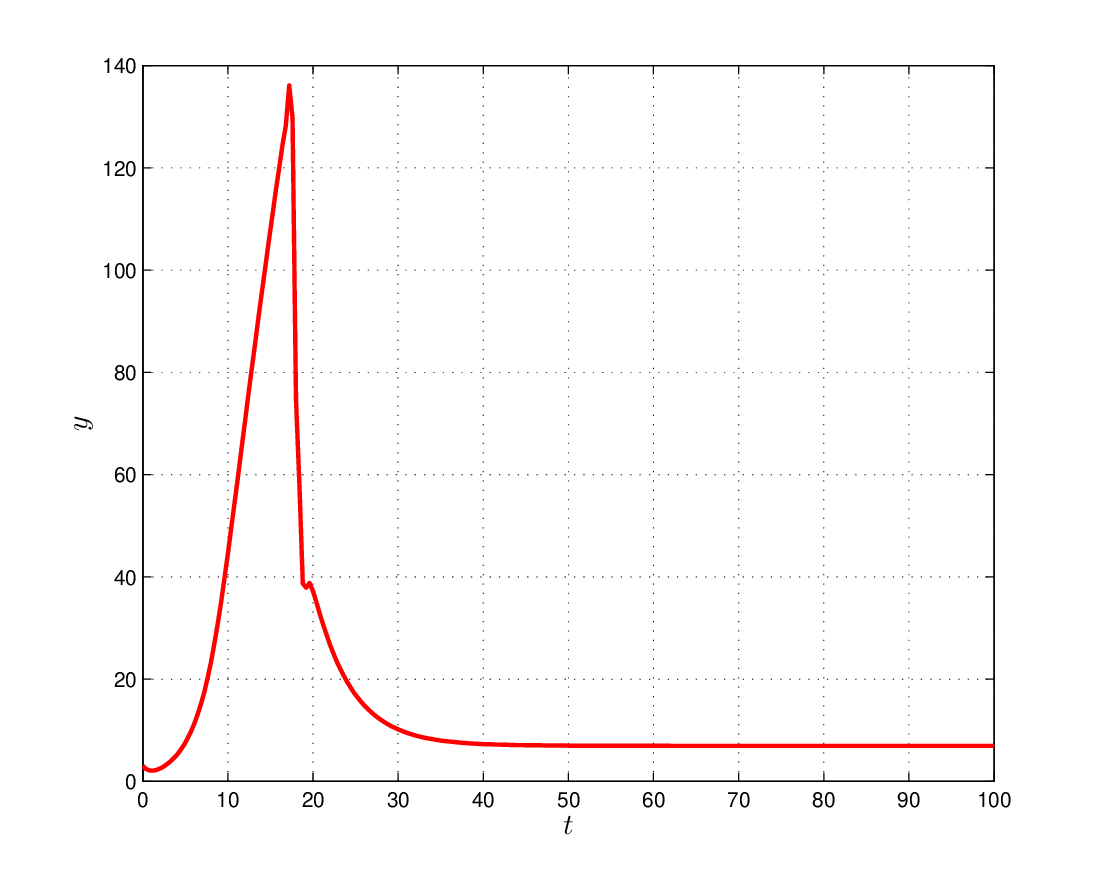}
\label{Figure:1b}
}\hfill
\caption{The approximate solution generated by the trapezoidal method with $h = 0.4$}\label{fig:1}
\end{figure}

\begin{figure}[H]
\subfloat[$x$-component]{%
\includegraphics[height=10cm,width=15cm]{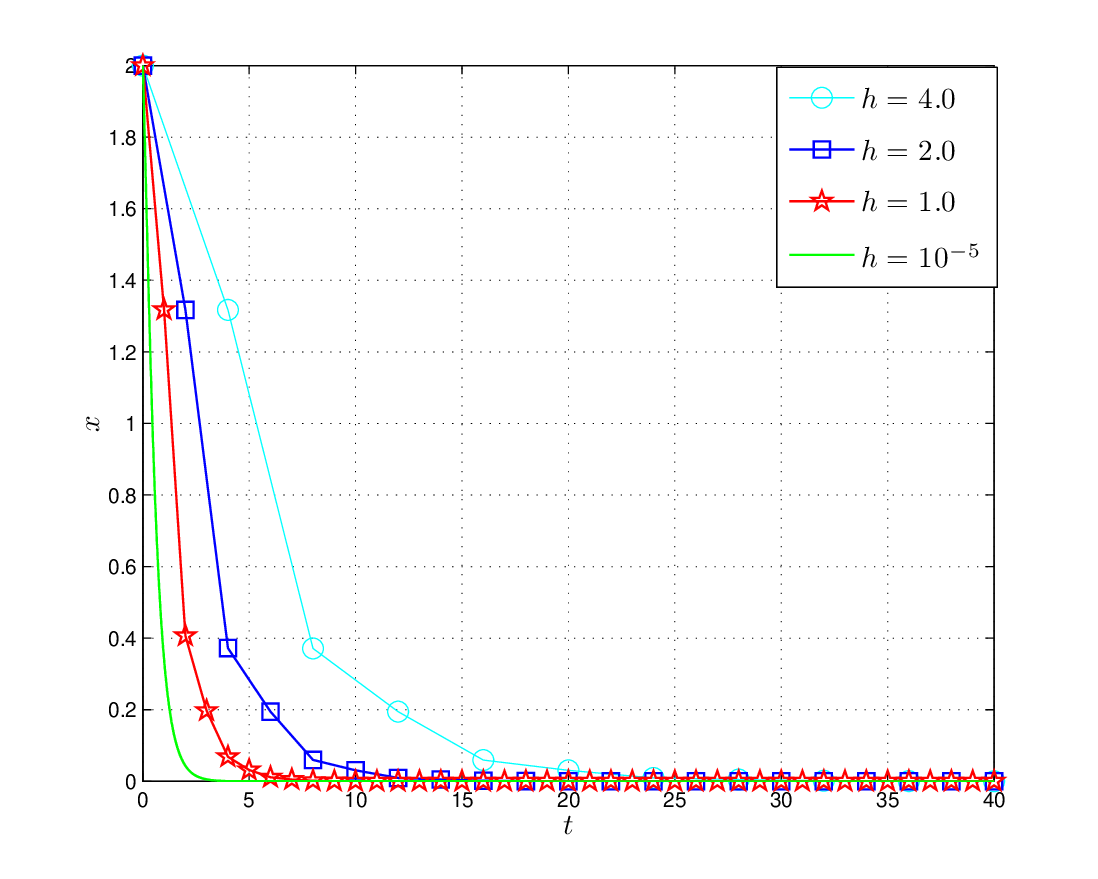}
\label{Figure:2a}
}\hfill
\subfloat[$y$-component]{%
\includegraphics[height=10cm,width=15cm]{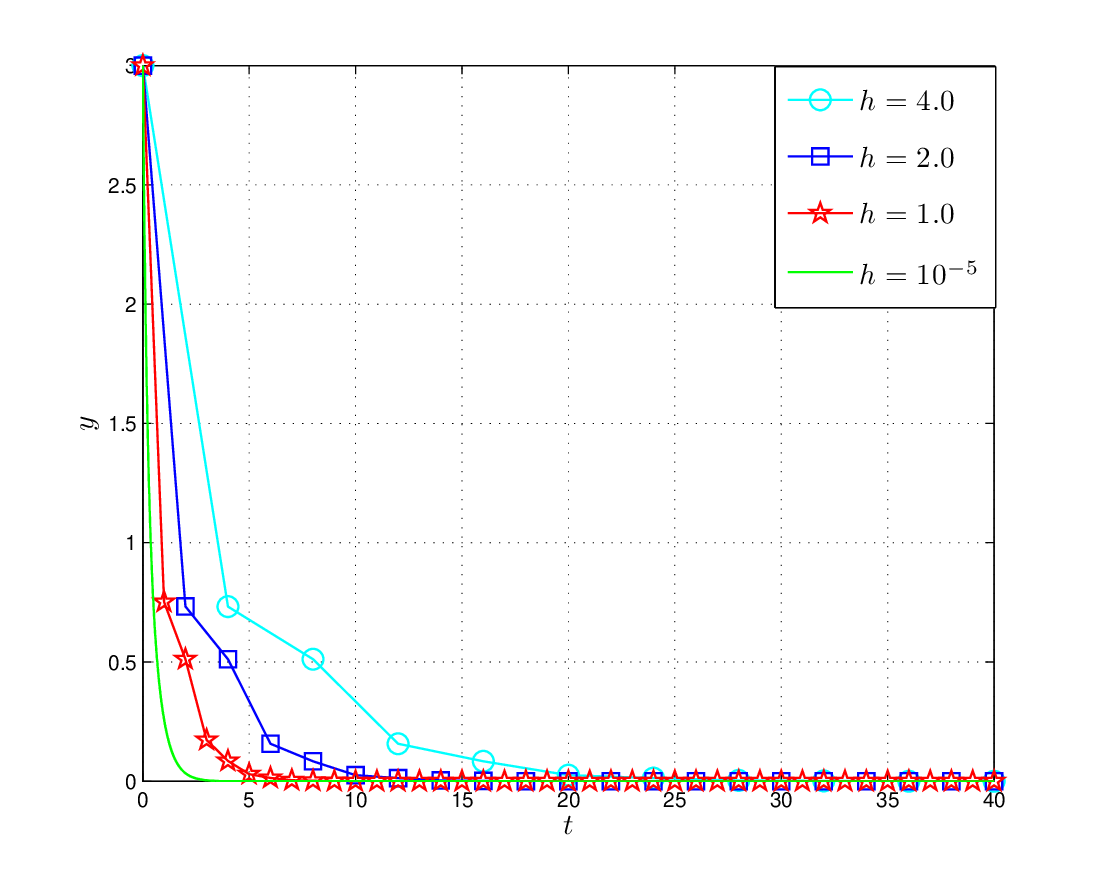}
\label{Figure:2b}
}\hfill
\caption{The approximate solution generated by the $(2.8970,\,4.5500)$-2ndNSFD scheme method with some different values of $h$}\label{fig:2}
\end{figure}
\end{example}
\begin{example}[Numerical dynamics of the second-order NSFD scheme]
In this example, we implement the NSFD scheme \eqref{eq:NSFD} to simulate the behaviour of the continuous-time model \eqref{eq:1} over long time periods. For this purpose, we consider \eqref{eq:1} with the parameters given in Table \ref{Table8}. Note that the values of $\tau_1$ and $\tau_2$ are easily determined from \eqref{eq:DC1new} and \eqref{eq:DC2new}.
\begin{table}[H]
\begin{center}
\caption{The parameters used in numerical simulation}\label{Table8}
\begin{tabular}{cccccccccc}
\hline
Set&$\alpha$&$\beta$&$\delta$&$\mu$&$F$&Source&$\mathcal{R}_0$&Stable equilibrium&$(\tau_1,\,\tau_2)$\\
\hline
$1$&20&60&14.6&0.897&3.6530&\cite{Ladino}&0.5178&$E_T^* = (0,\,0)$&$(1.2303,\,4.5500)$\\
$2$&20&60&14.6&0.63&0.75&\cite{Ladino}&8.6897&$E_P^* = (244.1178,\, 11.6334)$&$(0.9633,\, 1.3800)$\\
\hline
\end{tabular}
\end{center}
\end{table}
Approximate solutions obtained by employing \eqref{eq:NSFD} with $h = 0.0001$ are given in Figure \ref{fig:3}. It is clear that the results shown in this figure are consistent with the mathematical analysis presented in \cite{Ladino}. Therefore, the constructed NSFD scheme is simple and effective for simulating the dynamics of the continuous-time model over long time intervals, thanks to its dynamical consistency.
\begin{figure}[H]
\subfloat[Parameter Set 1]{%
\includegraphics[height=10cm,width=15cm]{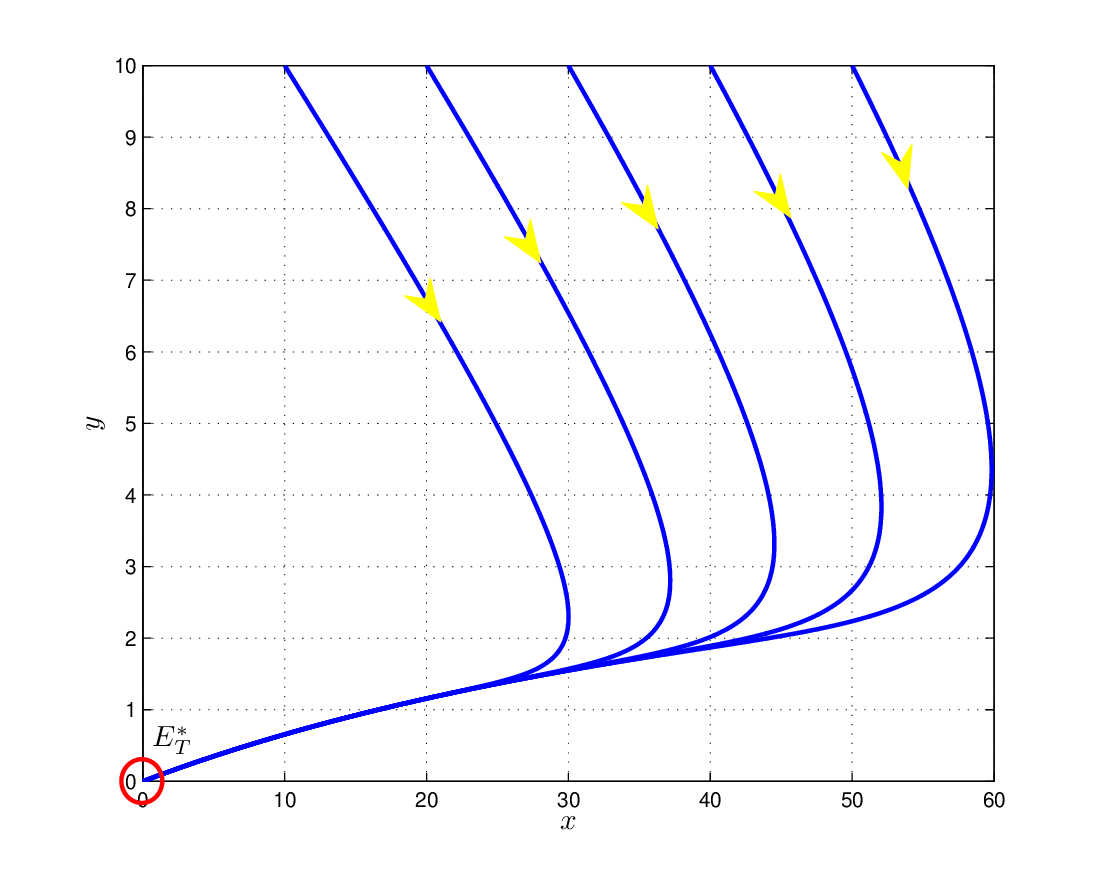}
\label{Figure:1a}
}\hfill
\subfloat[Parameter Set 2]{%
\includegraphics[height=10cm,width=15cm]{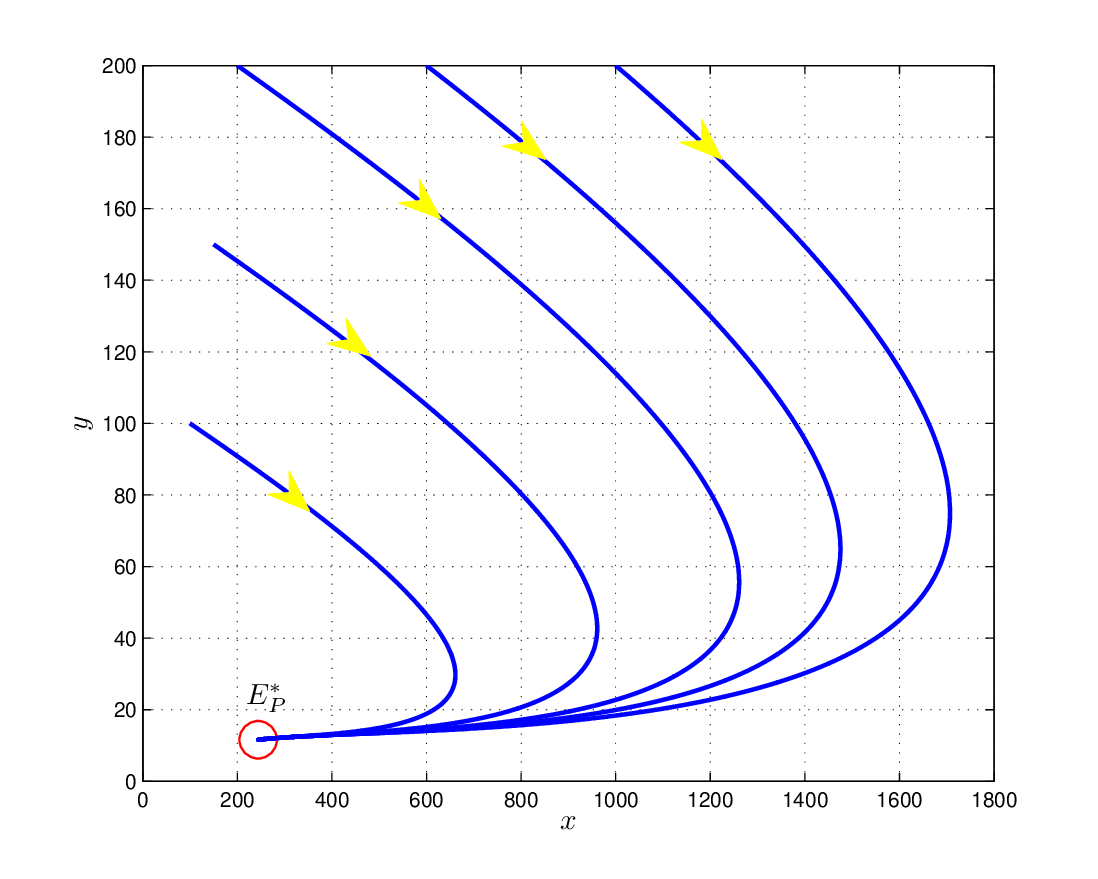}
\label{Figure:1b}
}\hfill
\caption{The phase planes of the two-stage structured species model generated by the second-order NSFD scheme}\label{fig:3}
\end{figure}
\end{example}
\section{Concluding remarks and discussions}
As the main conclusion of this work, we have constructed a simple second-order NSFD method, which adopts the approach in \cite{Hoang2023}, for a class of autonomous dynamical systems modeling various important phenomena and processes encountered in real-world situations. This method simultaneously preserves two properties of the continuous-time models for any finite step sizes, namely, the positivity of solutions, the set of equilibrium points and their asymptotic stability. The constructed NSFD method includes nonstandard denominator functions and a weighted discretization of the right-hand side functions. Here, the denominator functions guarantee second-order convergence and the weights ensure the dynamic consistency of the NSFD method. By taking the specific structure of the right-hand side functions, we have employed a simple discretization rather than the nonlocal approaches commonly used in previous studies. This simplifies the construction of the proposed NSFD method and facilitates the analysis of its asymptotic stability. Also, the second-order NSFD method can be readily combined with the Richardson extrapolation technique to improve its accuracy

As an illustration and an important application, we have applied the constructed second-order NSFD method to a well-known two-stage structured species model with recruitment, which was first proposed in \cite{Ladino}. Consequently, we derive a simple second-order NSFD scheme that improves upon a first-order NSFD scheme developed in \cite{HoangValverde}. Numerical experiments have been conducted to demonstrate the advantages of the NSFD scheme over a standard second-order method, namely, the explicit trapezoidal method.

The proposed approach is simple and can be applied to a broad class of dynamical system models arising in both theory and applications. In the near future, we will develop this approach to construct higher-order NSFD methods for partial differential equations and fractional-order differential equations.\\
\textbf{Ethical Approval:} Not applicable.\\
\textbf{Availability of supporting data:} The data supporting the findings of this study are available within the article [and/or] its supplementary materials.\\
\textbf{Conflicts of Interest:} The author declares no conflicts of interest to disclose.\\
\textbf{Authors' contributions:} Manh Tuan Hoang:  Writing review \& editing, Writing original draft, Visualization, Validation, Supervision, Software, Resources, Project administration, Methodology, Investigation,
Formal analysis, Data curation, Conceptualization, Funding acquisition.\\
\textbf{Funding information:} Not available.
\bibliographystyle{amsalpha}

\end{document}